\documentclass[preprint,12pt]{elsarticle}

\usepackage{amssymb}
\usepackage{amsmath,amssymb,amsfonts}
\usepackage{tikz}

\usetikzlibrary{decorations.pathmorphing, patterns, decorations.markings}

\newtheorem{thm}{Theorem}
\newtheorem{Cor}{Corollary}
\newtheorem{Lem}{Lemma}
\newdefinition{rmk}{Remark}
\newproof{proof}{Proof}
\journal{Mechanical Systems and Signal Processing}

\begin{document}

% Last update 2 May, 2026

\begin{frontmatter}

\title{Simultaneous determination of an unknown bending moment and shear force
in the Euler–Bernoulli cantilever beam from measured boundary deflection and slope}

\author[1]{Alemdar Hasanov\fnref{fn1}}
\ead{alemdar.hasanoglu@gmail.com}
\author[2]{Onur Baysal\fnref{fn2}\corref{cor1}
\fnref{fn2}}
\ead{onur.baysal@um.edu.mt}
\cortext[cor1]{Corresponding author}
\fntext[fn1]{Emeritus Professor. \c{S}ehit Ekrem Dsitrict, Altun\c{s}ehir Str., Ayazma Villalari, No: 22. Bah\c{c}ecik - Ba\c{s}iskele, Kocaeli, 41030}
\fntext[fn2]{Senior Lecturer, Department of Mathematics, University of Malta, Malta}
\address[1]{Kocaeli University, 41001, \.{I}zmit/Kocaeli, T\"{u}rkiye}
\address[2]{University of Malta, Msida, Malta}

\begin{abstract}
It is known that the study of vibration characteristics at the tip of the micro-cantilever and its relationship to the sample plays a very important role in improving the resolution of an Atomic Force Microscopy (AFM). In this paper, within the Euler-Bernoulli beam structure, a mathematical model, defined as \emph{a model with two unknown inputs and two measured outputs}, is considered for the simultaneous determination of the unknown bending moment and the shear force at the tip of the micro-cantilever from two feasible measured outputs at the same tip: the deflection and the slope. This model leads to the following inverse problem: find $M(t)$ and $g(t)$ in $\rho_A(x)u_{tt}+\mu(x) u_{t}+(r(x)u_{xx})_{xx}=0$,
$(x,t)\in \Omega_T:=(0,\ell)\times (0,T)$ subject to the boundary conditions $u(0,t)=u_{x}(0,t)=0$, $\left(r(x)u_{xx}\right)_{x=\ell}=M(t)$, $\left((r(x)u_{xx})_x\right)_{x=\ell}=g(t)$, and the homogenous initial conditions $u(x,0)=u_{t}(x,0)=0$, from the measured outputs  $w_{\ell}(t):=u(\ell,t)$ and $\theta_{\ell}(t):=u_x(\ell,t)$.   It is proved that the vector-form input-output map $\mathcal{P}:=\left (\Phi, \Psi \right )$, with $\left (\Phi q \right )(t):=u(0,t;q)$ and $\left (\Psi q\right )(t):=u_x(0,t;q)$, where $q(t):=\left (M(t),g(t)\right )$, corresponding to the inverse problem, is compact and Lipschitz continuous. This result allows us to prove the existence of a solution of the minimization problem for the Tikhonov functional $J(q)=\frac{1}{2}\Vert \Phi q-w_{\ell} \Vert_{L^2(0,T)}^2+\frac{1}{2}\Vert \Psi q-\theta_{\ell} \Vert_{L^2(0,T)}^2$. As a consequence, the existence of a quasi-solution to the inverse problem is established. Furthermore, a vector-form expression for the Fréchet gradient of the Tikhonov functional is derived, and the Lipschitz continuity of the Fréchet gradient is rigorously proven. This crucial property ensures the monotonic behavior of iterative gradient-based numerical methods. In addition, the obtained results are directly applicable to the analysis and design of both atomic force microscopes (AFMs) and their advanced counterparts, transverse dynamic force microscopes (TDFMs).

 %As a consequence, the existence of a quasi-solution to the inverse problem is proved. Further, a vector-form gradient formula for the Fr\'{e}chet gradient of the Tikhonov functional is derived. The Lipschitz continuity of the Fr\'{e}chet gradient is proved. This important property guarantees the monotonicity of the iterations in gradient methods. The obtained explicit gradient formula and the solution of the related adjoint problem together form the basis for numerical solution methods based on the conjugate gradient algorithm (CGA) for the numerical solution of the inverse problem. In addition, these results are applicable to the use and design of both atomic force microscopes (AFMs) and their advanced version, transverse dynamic force microscopes (TDFM).

\end{abstract}

\end{frontmatter}

%% \linenumbers

\section{Introduction}

Atomic force microscopy (AFM) is an indispensable tool in crucial processes such as imaging, manipulation, and lithography at the nanometer scale. Invented by Binning and coworkers in 1986 \cite{Binnig:1986}, this instrument has been subsequently improved through various enhancements \cite{Abramovitch:2007, Raman:2008}. In scanning processes for measuring surface topography and material properties with AFM, the most fundamental element is a sharp tip that moves across the surface of the sample. To understand the interaction forces between the tip and the sample, adequate mathematical modeling with complete vibration analysis of the AFM cantilever is necessary  \cite{Jalili:2012}. In this context, investigating the moment and shear force acting on the micro-cantilever tip is among the primary problems in studying the dynamic operation of AFM \cite{Jalili:2004}.

The tip is located at the free end of an AFM micro-cantilever, and vibration begins as a result of the dynamic interaction forces between this tip and the sample. This vibration creates a transverse shear force and a bending moment on the tip of the cantilever \cite{Haugstad:2012}. Estimating the unknown shear force and bending moment allows us to have a better interpretation and understanding of the scan results. However, these forces can only be measured indirectly through a laser-based sensor system. For this reason, the problem of determining the unknown bending moment and shear force from experimentally measured deflection or/and slope, based on a suitable mathematical model, has been one of the most challenging inverse problems that has been continuously studied for over a quarter of a century (see \cite{Antognozzi:2001, Chang:2004, Zhang:2022, Baysal:2023, Hasanov:Kawano:2022} and references therein).

A common feature of the previous studies is that almost all of them address inverse problems involving the determination of a single input, such as the shear force, from a single observed output, such as the measured deflection at the tip of a cantilever. These problems are classified as inverse problems that are governed by the Dirichlet-to-Neumann map, according to the commonly used terminology. These types of problems can also be specified as inverse problems based on one unknown input and one measured output. The inverse problem of determining both the bending moment $M(t)$ and the shear force $g(t)$ from the measured deflection at the end of the AFM cantilever is considered in \cite{AH:OB:AK:2024}. However, since a cone-shaped cantilever type is involved here, there is the relationship $M(t) = -\left (2h \cos \theta/\pi \right )\, g(t)$ between the inputs $M(t)$ and $g(t)$, where $h, \, \theta >0$ are the tip length and half-conic angle, respectively. Hence, the bending moment $M(t)$ here cannot be considered an independent unknown input.

In this study, a mathematical model within the Euler-Bernoulli theory is proposed for the simultaneous identification of two fundamental inputs, the bending moment and the shear force, affecting the tip of an AFM cantilever. 
 The most feasible data, the measured deflection and slope at the tip of the cantilever, are assumed as measured outputs. Since the deflection and slope at the tip of the AFM cantilever are experimentally measurable data, it is evident that the proposed model has a wide range of applications. In particular, the cantilever beam, examined in this article, whose free end is subjected to bending moment and shear force, has been the subject of many studies because it is an important part of numerous machine systems (see, for instance, \cite{Mohr:2010} and references therein). This inverse problem was first examined in\cite{CHH:CCS:07} for a simple beam model with constant coefficients and without damping term, where detailed numerical solution steps based on the Conjugate Gradient Algorithm (CGA) were presented. On the other hand, no further study has been reported in the literature for a more general beam model or mathematical analysis of the considered inverse problem. Our aim is to fill this gap. Therefore, the present study serves as a continuation of the work in \cite{CHH:CCS:07}.  

Note that the Euler-Bernoulli theory is the most appropriate model to predict the cantilever kinematics during snap-to-contact event \cite{Eppell:2022}.

The paper is structured as follows. The mathematical model of dynamic vibration resulting from the effect of the bending moment and the shear force at the tip of the AFM cantilever is described in Section 2. A priori estimates for the weak and regular weak solutions of the forward problem are derived in Section 3. In Section 4, the input-output operators corresponding to inverse problem are introduced. Here, some properties of these operators are also analyzed. In Section 5 the Tikhonov functional is introduced and existence of a quasi solution is proved. The Fr\'echet derivative in vector form of the  Tikhonov functional through a suitable adjoint problem is derived in Section 6.
In Section 7, the Lipschitz continuity of the Fr\'{e}chet gradient is proved.
%The numerical algorithm and the results of computational experiments related to the simultaneous reconstruction of the bending moment and shear force are given in Section 8. 
Some conclusions are discussed in the final Section 8.

%section 2
\section{Mathematical model of the dynamic vibration of the AFM micro-cantilever subjected to bending moment and shear force}

We use Euler-Bernoulli beam theory to derive the dynamic equation of the AFM micro-cantilever.
Within this theory, the vibration of the micro-cantilever subjected to bending moment and shear force is described by the following mathematical model:
\begin{eqnarray}\label{1}
\left\{ \begin{array}{ll}
\rho(x) u_{tt}+\mu(x)u_{t}+ \left (r(x)u_{xx}\right)_{xx} =0,~ (x,t)\in \Omega_{T}:=(0,\ell)\times (0,T);\\ [4pt]
u(x,0)=u_{t}(x,0)=0, ~x \in (0,\ell); \\ [4pt]
u(0,t)=u_{t}(x,0)=0, \,\left(r(x)u_{xx}\right)_{x=\ell}=M(t),\,
\left (-(r(x)u_{xx})_x\right)_{x=\ell}=g(t),\,t \in [0,T].
\end{array} \right.
\end{eqnarray}
Here and below, $u(x,t)$ is the transverse deflection in position $x\in (0,\ell)$ and at the time $t\in \times [0,T]$, while $T>0$ is the final time instance, which may be small enough, and $\ell>0$ is the length of the cantilever. Further, $\rho(x)>0$ and $r(x):=E(x)I(x)>0$ are the mass density and the flexural rigidity (or bending stiffness) of a nonhomogeneous beam, respectively, while $E(x)>0$ is the elasticity modulus and $I(x)>0$ is the moment of inertia. The external damping mechanism is given by the term $\mu(x)u_t$. The coefficient $\mu(x)\ge 0$ is called the viscous (internal) damping parameter.

The functions $M(t)$ and $g(t)$ represent the bending moment and shear force acting on the tip of the cantilever.

The deflection $w_\ell(t)$ and the slope $\theta_\ell(t)$ at the tip $x=\ell$ of the cantilever are assumed to be feasible measurable data:
\begin{eqnarray}\label{2}
\left. \begin{array}{ll}
w_\ell(t):=u(0,t),~ t \in [0,T],\\ [4pt]
\theta_\ell(t):=u_x(\ell,t),~ t \in [0,T].
\end{array} \right.
\end{eqnarray}

Within the framework of the mathematical model (\ref{1})-(\ref{2}), the inverse problem of simultaneous determining the unknown  bending moment $M(t)$ and shear force $g(t)$ based on the data $w_\ell(t)$ and $\theta_\ell(t)$ is defined as follows:\\
\emph{Find the unknown inputs $M(t)$ and $g(t)$ in (\ref{1}) from the measured outputs $w_\ell(t)$ and $\theta_\ell(t)$ introduced in (\ref{2}).}

The problem (\ref{1})-(\ref{2}) can also be defined as an inverse boundary value problem with two Neumann inputs $M(t)$ and $g(t)$ and two Dirichlet measured outputs $w_\ell(t)$ and $\theta_\ell(t)$, according to generally accepted terminology.

%Figure 1
  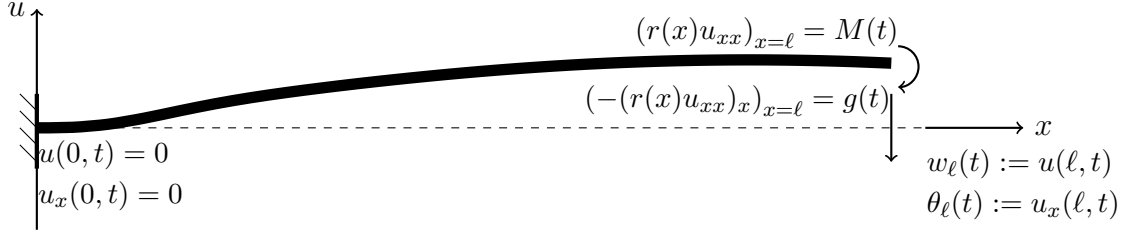
\begin{figure}
  $\quad$\\
  $\quad$\\
  \centering
 \hspace*{-.6cm} \begin{tikzpicture}[scale=.9]
      \draw[color=black,ultra thick] (0,-0.6) -- (0,0.5);
      \draw[color=black] (-0.25,-1+3*0.25) -- (0,-1+2*0.25);
      \draw[color=black] (-0.25,-1+4*0.25) -- (0,-1+3*0.25);
      \draw[color=black] (-0.25,-1+5*0.25) -- (0,-1+4*0.25);
      \draw[color=black] (-0.25,-1+6*0.25) -- (0,-1+5*0.25);
 %    tip of the micro-cantilever
     \draw[thick, ->] (4*3.14+0.0,0.5) -- (4*3.14+0.0,-0.5);
  %   \draw[dotted] (4*3.14,{sin(0.35*deg(4*3.14))/1.2}) -- (2*3.14,0);
      \draw[dashed] (0,0) -- (4*3.14+0.5,0);
      \draw[thick,->] (4*3.14+0.5,0) -- (14.5,0) node [right]{$x$};
      \draw[thick,->] (0,-1.5) -- (0,1.75) node [left]{$u$};
   \draw[color=black,line width=1.5mm,smooth,domain=0:{4*3.14}]
        plot (\x,{(1-exp(-1.5*\x*\x/4))*1.5*sin(0.15*deg(\x))/1.5});
 %  \draw [line width=1.5mm] (0,0) -- (4*3.14,0);
      %\draw[color=black,ultra thick,smooth,domain=0:{4*3.14}] plot (\x,{sin(0.35*deg(\x))/1.5});
      \node[label=right:{\small$ u(0,t)=0$}] at (-.3,-0.4) {};
      \node[label=right:{\small$ u_{x}(0,t)=0$}] at (-.3,-1.0) {};
      \node[label=left:{\small $\left ( r(x)u_{xx}\right )_{x=\ell}=M(t)$}] at (4.20*3.14-0.2,1.4) {};
      \node[label=left:{\small$\left (-(r(x)u_{xx})_x\right )_{x=\ell}=g(t)$}] at (4.10*3.14,0.4) {};
      \draw[thick, <-] (4*3.14+0.1,0.6) arc (-90:90:0.3);
   \node[label=right:{\small$ w_{\ell}(t):=u(\ell,t)$}] at (4*3.14+0.2,{sin(0.35*deg(4*3.14))/1.5+0.1}) {};
   \node[label=right:{\small$ \theta_{\ell}(t):=u_x(\ell,t)$}] at (4*3.14+0.2,{sin(0.35*deg(4*3.14))/1.5-0.5}) {};
      \end{tikzpicture}
  \caption{Micro-cantilever subjected to bending moment and shear force}
  \label{Fig-1}
  \end{figure}

For given inputs $M(t)$ and $g(t)$  from some class of admissible functions, the initial boundary value problem (\ref{1}) will be referred as the \emph{forward problem}.

A schematic representation of the system governed by the model (\ref{1})-(\ref{2}) is given in Figure \ref{Fig-1}.

%section 2
\section{Analysis of the forward problem: necessary a priori estimates}

We derive here some necessary a priori estimates for the weak and regular weak solutions the forward problem (\ref{1}), employing the methodology given in \cite{Hasanov:Romanov:2021}.

We assume that the inputs and outputs in (\ref{1}) satisfy the following basic conditions:
\begin{eqnarray} \label{3}
\left \{ \begin{array}{ll}
\rho, r,\mu \in L^{\infty}(0,\ell),\\ [3pt]
M,g \in H^1(0,T),\\ [3pt]
0 < \rho_0 \le \rho(x)\le \rho_1,~0< r_0 \le r(x) \le r_1, ~\mu(x) \ge 0, ~ x \in (0,\ell).
\end{array} \right.
\end{eqnarray}

From the theory developed in \cite{Hasanov:Romanov:2021, Baysal:2019} it follows that under conditions (\ref{3}) there exists a unique weak solution $u\in L^2(0,T;\mathcal{V}^2(0,\ell))$ with $u_t\in L^2(0,T;L^2(0,\ell))$, $u_{tt}\in L^2(0,T;H^{-2}(0,\ell))$ of the forward problem (\ref{1}). Here and below, $\mathcal{V}^2(0,\ell):=\{v\in H^2(0,\ell):~ v(0)=v'(0)=0\}$ is the subspace of the Sobolev space $H^2(0,\ell)$ \cite{Adams:1978}.

\begin{thm}\label{Theorem-1}
Assume that the basic conditions (\ref{3}) hold. Then for the weak solution of the forward problem (\ref{1}), the following estimates hold:
\begin{eqnarray}\label{4}
\left. \begin{array}{ll}
\displaystyle \Vert u_{xx} \Vert^2_{L^2(0,T;L^2(0,\ell))}
\leq  C_1^2 \, \Vert q \Vert^2_{\mathbb{H}^1(0,T)}\,,  \\ [8pt]
\displaystyle \Vert u_t \Vert^2_{L^2(0,T;L^2(0,\ell))}
\leq C_2^2 \, \Vert q \Vert^2_{\mathbb{H}^1(0,T)}\,,
\end{array} \right.
\end{eqnarray}
where $q(t):=(M(t),g(t))$ and
\begin{eqnarray}\label{5}
\left. \begin{array}{ll}
 \Vert q \Vert^2_{\mathbb{H}^1(0,T)}:=  \Vert M \Vert^2_{H^1(0,T)}+ \Vert g \Vert^2_{H^1(0,T)}\,,
  \\ [8pt]
\displaystyle C_1^2=\left [\exp(T)-1\right ] C_0^2,~C_2^2=\frac{r_0}{2 \rho_0} \,\left [C_0^2+C_1^2\right ]\,,
  \\ [12pt]
\displaystyle C_0^2=\frac{\exp(T)\,r_0}{2 \rho_0} \,C_T^2,~ C^2_T=\max \left (2/T,\, 1+2 T \right )
\end{array} \right.
\end{eqnarray}
and $\rho_0$ and $r_0>0$ are the constants introduced in (\ref{3}).
\end{thm}
{\bf Proof.} Multiply both sides of equation (\ref{1}) by $2u_t(x,t)$, integrate over $\Omega_t:=(0,\ell)\times (0,t)$, $t\in (0,T)$ and use the following identity
\begin{eqnarray}\label{6}
\left. \begin{array}{ll}
\displaystyle 2 \int_0^t \int_0^\ell \left (r(x)u_{xx}\right )_{xx} u_{\tau} dx d\tau \\ [10pt]
\qquad =
\displaystyle 2 \int_0^t \int_0^\ell \left [(r(x)u_{xx})_x u_{\tau}-r(x)u_{xx} u_{x\tau}\right ]_x dx d\tau
+  \int_0^t \int_0^\ell r(x) \left (u_{xx}^2\right )_{\tau} dx d\tau,
\end{array} \right.
\end{eqnarray}
for all $t\in[0,T]$. With the initial and boundary conditions in (\ref{1}), we obtain the following \emph{energy identity}:
\begin{eqnarray}\label{7}
\int_0^\ell \left [\rho(x) u_t^2 +r(x)u_{xx}^2\right ]dx+2 \int_0^t \int_0^\ell \mu (x)u_{\tau}^2 dx d \tau \qquad \qquad \qquad \qquad
 \nonumber \\ [1pt]
\qquad \qquad  = 2\int_0^t M(\tau) u_{x\tau}(\ell,\tau) d \tau+2\int_0^t g(\tau) u_{\tau}(\ell,\tau) d \tau,~t\in[0,T].
\end{eqnarray}
We transform and then estimate the right-hand-side integrals in (\ref{7}). We have,
\begin{eqnarray}\label{8}
2\int_0^t M(\tau) u_{x\tau}(\ell,\tau)d \tau =- 2\int_0^t M'(\tau) u_{x}(\ell,\tau)dx d \tau +
2 \left ( M(\tau) u_{x}(\ell,\tau) \right )_{\tau=0}^{\tau=t} \qquad \qquad
 \nonumber \\ [1pt]
\qquad \qquad  \displaystyle \le \varepsilon \, \left [\int_0^t u^2_{x}(\ell,\tau)d \tau + u^2_{x}(\ell,t) \right]+\frac{1}{\varepsilon}\,\left [\int_0^T \left (M'(t) \right)^2dt + M^2(t) \right],~t\in[0,T],
\end{eqnarray}
and
\begin{eqnarray}\label{9}
\hspace*{-2cm} 2\int_0^t g(\tau) u_{\tau}(\ell,\tau)d \tau =- 2\int_0^t g'(\tau) u(\ell,\tau)d \tau +
2 \left ( g(\tau) u(\ell,\tau) \right )_{\tau=0}^{\tau=t} \qquad \qquad\qquad  \ \ \ \
 \nonumber \\ [1pt]
\qquad \qquad  \displaystyle \le \varepsilon \, \left [\int_0^t u^2(\ell,\tau)d \tau + u^2(\ell,t) \right]+\frac{1}{\varepsilon}\,\left [\int_0^T \left (g'(t) \right)^2dt + g^2(t) \right],~t\in[0,T].
\end{eqnarray}
Here we need the following auxiliary trace inequalities
\begin{eqnarray}\label{10}
\left. \begin{array}{ll}
\displaystyle u^2(\ell,t) \le \frac{\ell^3}{2} \, \int_0^\ell u^2_{xx}(x,t)dx,~ t\in (0,T),
  \\ [10pt]
\displaystyle u_x^2(\ell,t) \le \ell \, \int_0^\ell u^2_{xx}(x,t)dx,~ t\in (0,T),
\end{array} \right.
\end{eqnarray}
for the weak solution of the forward problem (\ref{1}). These inequalities follow from the identity
\begin{eqnarray*}
u(\ell,t) = \int_0^\ell \int_0^x u_{\xi \xi}(\xi,t) \,d\xi\, d x.
\end{eqnarray*}

Now, we estimate the terms in parentheses on the right side of inequalities (\ref{8}) and (\ref{9}). To this end, we employ the mean value theorem for integrals,
 \begin{eqnarray*}
\exists s\in [0,T],~\displaystyle M(s)=\frac{1}{T}\int_0^T M(t) dt,
\end{eqnarray*}
for $M\in H^1(0,T)$, and the identity
 \begin{eqnarray*}
\displaystyle M(t)=M(s)+\int_s^t M(\tau) d\tau,\, t\in [0,T].
\end{eqnarray*}
In view of the Cauchy-Schwarz inequality we deduce that
\begin{eqnarray*}
|M(t)| \leq |M(s)| + \left|\int_s^t M'(\tau) d\tau \right| \leq  \frac{1}{\sqrt{T}}\, ||M||_{L^2(0,T)} + \sqrt{T} \int_0^T| M'(t)|^2 dt.
\end{eqnarray*}
 Taking the maximum over [0,T], left hand side becomes the  $C[0,T]$-norm. Similar estimate can be deduced for $g\in H^1(0,T)$. As a consequence, we obtain following one-dimensional forms
\begin{eqnarray}\label{11}
\left. \begin{array}{ll}
\displaystyle \Vert M \Vert_{C[0,T]} \le \frac{1}{\sqrt{T}} \,\Vert M \Vert_{L^2(0,T)}+\sqrt{T} \, \,\Vert M' \Vert_{L^2(0,T)},\\ [10pt]
\displaystyle \Vert g \Vert_{C[0,T]} \le \frac{1}{\sqrt{T}} \,\Vert g \Vert_{L^2(0,T)}+\sqrt{T}\, \,\Vert g' \Vert_{L^2(0,T)}
\end{array} \right.
\end{eqnarray}
of the Sobolev embedding estimates \cite{Adams:1978}.

Considering inequalities (\ref{10}) and (\ref{11}) in (\ref{8}) and (\ref{9}), we obtain:
\begin{eqnarray*}
\displaystyle 2\int_0^t M(\tau) u_{x\tau}(\ell,\tau)d \tau \le \ell \varepsilon\left [ \int_0^t \int_0^\ell u^2_{xx}(x,\tau)dx d\tau+\int_0^\ell u^2_{xx}(x,t)dx \right ] +\frac{C^2_T}{\varepsilon}\,\Vert M \Vert^2_{H^1(0,T)} \\ [4pt]
\displaystyle 2\int_0^t g(\tau) u_{\tau}(\ell,\tau)d \tau \le \frac{\ell^3 \varepsilon}{2}\left [ \int_0^t \int_0^\ell u^2_{xx}(x,\tau)dx d\tau+\int_0^\ell u^2_{xx}(x,t)dx \right ]+ \frac{C^2_T}{\varepsilon} \,\Vert g\Vert^2_{H^1(0,T)},
\end{eqnarray*}
for all $t\in[0,T]$, with the constant $C_T>0$ introduced in (\ref{5}). Then using these inequalities in the energy identity in (\ref{7}) we deduce that
\begin{eqnarray}\label{12}
\displaystyle \rho_0 \int_0^\ell u_t^2dx  +\left (r_0 -(\ell+\ell^3/2)\varepsilon \right ) \int_0^\ell u_{xx}^2dx+2 \int_0^t \int_0^\ell \mu (x)u_{\tau}^2\, dx d \tau \qquad \quad \qquad \quad
 \nonumber \\ [3pt]
\displaystyle \le \left (\ell+\ell^3/2\right )\varepsilon \, \int_0^t \int_0^\ell u_{xx}^2dx d\tau
+\frac{C^2_T}{\varepsilon} \,\Vert q\Vert^2_{\mathbb{H}^1(0,T)},~t\in[0,T]. \qquad
\end{eqnarray}
We define the arbitrary parameter $\varepsilon>0$ as follows:
\begin{eqnarray*}
\varepsilon = \frac{r_0 }{ 2\ell+\ell^3}  .
\end{eqnarray*}
Then $r_0-\left (\ell+\ell^3/2\right)\varepsilon=r_0/2$ and (\ref{12}) implies the main integral inequality:
\begin{eqnarray}\label{13}
\displaystyle \rho_0 \int_0^\ell u_t^2dx  +\frac{r_0 }{2} \int_0^\ell u_{xx}^2dx+2 \int_0^t \int_0^\ell \mu (x)u_{\tau}^2 dx d \tau  \qquad \qquad \qquad \qquad \nonumber \\ [3pt]
\displaystyle \le \frac{r_0 }{2} \int_0^t \int_0^\ell u_{xx}^2dx d\tau +\frac{\left (2\ell+\ell^3\right)C^2_T}{r_0} \, \Vert q\Vert^2_{\mathbb{H}^1(0,T)},~t\in[0,T], \quad
\end{eqnarray}
where the right-hand-side norm $\Vert q \Vert^2_{\mathbb{H}^1(0,T)}$ is defined in (\ref{5}).

As a first consequence of (\ref{13}), we find:
\begin{eqnarray*}
\int_0^\ell u_{xx}^2dx  \le  \int_0^t \int_0^\ell u_{xx}^2dx d\tau +C_0^2\, \Vert q \Vert^2_{\mathbb{H}^1(0,T)},~t\in[0,T], \quad
\end{eqnarray*}
where $C_0>0$ is the constant introduced in (\ref{5}). With the Gronwall-Bellman inequality, this yields:
\begin{eqnarray}\label{14}
 \int_0^\ell u_{xx}^2dx  \le  C_0^2 \Vert q \Vert^2_{\mathbb{H}^1(0,T)} \exp(t),~t\in[0,T]. \quad
\end{eqnarray}
Integrating this inequality over $[0,T]$ we arrive at the first estimate in (\ref{4}).

The second consequence of (\ref{13}) is the following inequality:
\begin{eqnarray*}
\int_0^\ell u_t^2dx  \le  \frac{r_0}{2\rho_0} \int_0^t \int_0^\ell u_{xx}^2dx d\tau +\frac{\left (2\ell+\ell^3\right)C^2_T}{\rho_0 r_0} \, \Vert q \Vert^2_{\mathbb{H}^1(0,T)},~t\in[0,T]. \quad
\end{eqnarray*}
With inequality (\ref{14}) this implies the second  estimate in (\ref{4}).  \hfill$\Box$

\begin{thm}\label{Theorem-2}
Assume that in addition to the conditions of Theorem \ref{Theorem-1}, the following regularity conditions hold:
\begin{eqnarray} \label{15}
r \in H^2(0,\ell),\, M,g \in H^2(0,T).
\end{eqnarray}
Then for the regular weak solution $u\in L^2(0,T;H^4(0,\ell))$, $u_t\in L^2(0,T;\mathcal{V}^2(0,\ell))$, $u_{tt}\in L^2(0,T;L^2(0,\ell))$ and $u_{ttt}\in L^2(0,T;H^{-2}(0,\ell))$ of the forward problem (\ref{1}), the following a priori estimates hold:
\begin{eqnarray}\label{16}
\left. \begin{array}{ll}
\displaystyle \Vert u_{xxt} \Vert^2_{L^2(0,T;L^2(0,\ell))}
\leq  C_1^2 \, \Vert q\Vert^2_{\mathbb{H}^2(0,T)}\,,  \\ [10pt]
\displaystyle \Vert u_{tt} \Vert^2_{L^2(0,T;L^2(0,\ell))}
\leq C_2^2\,  \Vert q \Vert^2_{\mathbb{H}^2(0,T)}\,,
\end{array} \right.
\end{eqnarray}
where
\begin{eqnarray}\label{17}
 \Vert q \Vert^2_{\mathbb{H}^2(0,T)}:=\Vert M \Vert^2_{H^2(0,T)}+ \Vert g \Vert^2_{H^2(0,T)}
\end{eqnarray}
and $C_1,C_2>0$ are the constants introduced in (\ref{5}).
\end{thm}
{\bf Proof.} We first differentiate equation (\ref{1}) with respect to \(t\), then multiply both sides by $2u_{tt}(x,t)$, integrate over $\Omega_t:=(0,\ell)\times (0,t)$, and then use the analogue of the identity (\ref{6}) with $u_t$ replaced by $u_tt$. In view of the initial and boundary conditions in (\ref{2}), we obtain the following integral identity:
\begin{eqnarray*}
\int_0^\ell \left [\rho(x) u_{tt}^2 +r(x)u_{xx\tau}^2\right ]dx+2 \int_0^t \int_0^\ell \mu (x)u_{\tau\tau}^2 dx d \tau  \qquad \qquad \qquad  \qquad\qquad \qquad
 \nonumber \\ [1pt]
\qquad \qquad \qquad \qquad  = 2\int_0^t M'(\tau) u_{x\tau\tau}(\ell,\tau) d \tau+2\int_0^t g'(\tau) u_{\tau\tau}(\ell,\tau) d \tau,~t\in[0,T].
\end{eqnarray*}
Transforming the right-hand-side integrals in the same way as in the proof of Theorem 1, we obtain the following integral inequality:
\begin{eqnarray*}
\displaystyle \rho_0 \int_0^\ell u_{tt}^2dx  +\frac{r_0 }{2} \int_0^\ell u_{xxt}^2dx+2 \int_0^t \int_0^\ell \mu (x)u_{\tau\tau}^2 dx d \tau  \qquad \quad \qquad \qquad \qquad \nonumber \\ [3pt]
\displaystyle \le \frac{r_0 }{2} \int_0^t \int_0^\ell u_{xxt}^2dx d\tau +\frac{\left (2\ell+\ell^3\right)C^2_T}{r_0} \,\Vert q \Vert^2_{\mathbb{H}^2(0,T)},~t\in[0,T], \quad
\end{eqnarray*}
where the norm $\Vert q\Vert^2_{\mathbb{H}^2(0,T)}$ on right-hand-side is defined in (\ref{17}).

The required estimates in (\ref{16}) are derived from this inequality.  \hfill$\Box$

The following trace estimates are the consequence of the estimates (\ref{4}), (\ref{10})  and (\ref{16}).

\begin{Cor}\label{Corollary-1}
The following trace estimates
\begin{eqnarray}\label{18}
\left. \begin{array}{ll}
\displaystyle \Vert u(\ell,\cdot) \Vert^2_{L^2(0,T)}
\leq   \frac{\ell^3 }{2}\, C_1^2 \, \Vert q \Vert^2_{\mathbb{H}^1(0,T)}\,, \\ [10pt]
\displaystyle \Vert u_{x} (\ell,\cdot) \Vert^2_{L^2(0,T)}
\leq \ell \,C_1^2 \, \Vert q \Vert^2_{\mathbb{H}^1(0,T)}
\end{array} \right.
\end{eqnarray}
and
\begin{eqnarray}\label{19}
\left. \begin{array}{ll}
\displaystyle \Vert u_t(\ell,\cdot) \Vert^2_{L^2(0,T)}
\leq   \frac{\ell^3 }{2}\, C_1^2 \, \Vert q \Vert^2_{\mathbb{H}^2(0,T)}\,,  \\ [10pt]
\displaystyle \Vert u_{xt} (\ell,\cdot) \Vert^2_{L^2(0,T)}
\leq \ell \,C_1^2 \, \Vert q \Vert^2_{\mathbb{H}^2(0,T)}
\end{array} \right.
\end{eqnarray}
hold with the same constant $C_1>0$ introduced in (\ref{5}), for the weak and regular weak solutions of the forward problem (\ref{1}), under the conditions of Theorem \ref{Theorem-1}, and Theorem \ref{Theorem-2}, respectively.
\end{Cor}

%section 4
\section{The input-output operators}

We introduce the \emph{sets of admissible inputs} as follows:
\begin{eqnarray} \label{20}
\left. \begin{array}{ll}
\mathcal{M}(0,T)=\{M\in H^2(0,T):~ \Vert M \Vert_{H^2(0,T)}\le C_\mathcal{Q}\}, \\ [6pt]
\mathcal{G}(0,T)=\{g\in H^2(0,T):~ \Vert g \Vert_{H^2(0,T)}\le C_\mathcal{Q}\},
\end{array} \right.
\end{eqnarray}
where $C_\mathcal{Q}>0$ is a constant independent of $M(t)$ and $g(t)$. Further, we define set $\mathcal{Q}(0,T) \subset H^2(0,T)\times H^2(0,T)$ consisting of pairs of sets of admissible inputs:
\begin{eqnarray} \label{21}
\mathcal{Q}(0,T):=\{q(t):=\langle M(t), g(t) \rangle : \, M\in\mathcal{M}(0,T), \, g\in\mathcal{G}(0,T) \},
\end{eqnarray}
where $\mathcal{M}(0,T)$ and $\mathcal{G}(0,T)$ are the sets of admissible inputs introduced in (20).

Denote by $u(x,t;q)$ the unique weak solution of the forward problem (\ref{1}) corresponding to this input $q\in \mathcal{Q}(0,T)$. Then $u(\ell,t;q)$ and $u_x(\ell,t;q)$ are the outputs. This naturally leads to introducing the following input-output operators:
\begin{eqnarray}\label{22}
\left. \begin{array}{ll}
\Phi(q)(t):=u(\ell,t;q),~q \in \mathcal{Q}(0,T),~\Phi: q \mapsto L^2(0,T); \\[5pt]
\Psi(q)(t):=u_x(\ell,t;q),~q \in \mathcal{Q}(0,T),~\Psi: q \mapsto L^2(0,T),
\end{array} \right.
\end{eqnarray}
where $\mathcal{Q}(0,T)$ is the set introduced in (\ref{21}).

According to the widely used terminology of inverse problems, these input-output operators (\ref{22}) can alternatively be described as Neumann-to-Dirichlet maps.

In view of the above input-output operators, the inverse problem (\ref{1})-(\ref{2}) can be reformulated as the following system of operator equations:
 \begin{eqnarray}\label{23}
\left \{ \begin{array}{ll}
\Phi(q)(t)=w_\ell(t),~w_\ell \in  L^2(0,T), \\ [4pt]
\Psi(q)(t)=\theta_\ell(t),~\theta_\ell \in  L^2(0,T),~q \in \mathcal{Q}(0,T).
\end{array} \right.
\end{eqnarray}

From the following lemma it follows that the considered inverse problem is ill-posed.

\begin{Lem}\label{Lemma-1}
Assume that conditions of Theorem \ref{Theorem-2} are satisfied. Then the input-output operators introduced in (\ref{22}) are compact operators.
\end{Lem}
{\bf Proof.} Let $\{q^{(m)} \}$ be a bounded sequence in  $\mathcal{Q}(0,T)$. Therefore, inputs $\{M^{(m)} \}\subset \mathcal{M}(0,T)$ and $\{g^{(m)} \}\subset \mathcal{G}(0,T)$ are  bounded  sequences in $H^2(0,T)$. Denote by $\{\Phi(q^{(m)}) \}$ and $\{\Psi(q^{(m)}) \}$ the sequences of outputs: $\Phi(q^{(m)})(t)=u(\ell,t;q^{(m)})$ and $\Psi(q^{(m)})(t)=u_x(\ell,t;q^{(m)})$, $m=1,2,3,\, ...~ $.
From the estimates (\ref{18}) and (\ref{19}), it follows that these sequences are bounded in the norm of the Sobolev space $H^1(0,T)$, and hence compact in $L^2(0,T)$ due to Rellich-Kandrachov compactness theorem \cite{Adams:1978}. This means that the  input-output operators transform  bounded sequences $\{M^{(m)} \}$ and $\{g^{(m)}\}$  in $H^2(0,T)$  to the compact sequences $\{\Phi(M^{(m)}) \}$, and $\{\Psi(g^{(m)}) \}$ in $L^2(0,T)$, respectively. Hence, these operators are compact. \hfill$\Box$

Next lemma shows that the input-output operators are also Lipschitz continuous.

\begin{Lem}\label{Lemma-2}
Under the conditions of Theorem \ref{Theorem-1}, the input-output operators introduced in (\ref{21}) are Lipschitz continuous, that is,
\begin{eqnarray}\label{24}
\left. \begin{array}{ll}
\displaystyle \Vert \Phi(q_1)-\Phi(q_2) \Vert_{L^2(0,T)}
\leq  L_1 \Vert q_1-q_2 \Vert_{\mathbb{H}^1(0,T)},  \\ [8pt]
\displaystyle \Vert \Psi(q_1)-\Psi(q_2) \Vert_{L^2(0,T)}
\leq  L_2 \Vert q_1-q_2 \Vert_{\mathbb{H}^1(0,T)}, ~\forall q_1,q_2 \in \mathcal{Q}(0,T),
\end{array} \right.
\end{eqnarray}
where $L_1=\ell \sqrt{\ell}\,C_1/\sqrt{2},\,L_2=\sqrt{\ell}\,C_1 >0$ are the Lipschitz constants, and $C_1>0$ is the constant introduced in (\ref{6}).
\end{Lem}
{\bf Proof.} By the definition (\ref{22}) of the input-output operators,
\begin{eqnarray*}
\left. \begin{array}{ll}
\displaystyle \Phi (q_1)(t)-\Phi (q_2) (t)=\delta u(\ell,t),  \\ [8pt]
\displaystyle \Psi (q_1)(t)-\Psi (q_2)(t)=\delta u_x(\ell,t),
\end{array} \right.
\end{eqnarray*}
where $\delta u(\ell,t)=u(\ell,t;q_1)-u(\ell,t;q_2)$, with $q_1,q_2 \in \mathcal{Q}(0,T)$, is the solution of the following initial boundary value problem:
\begin{eqnarray}\label{25}
\left\{ \begin{array}{ll}
\rho(x) \delta u_{tt}+\mu(x)\delta u_{t}+ \left (r(x)\delta u_{xx}\right)_{xx} =0,~ (x,t)\in \Omega_{T};\\ [4pt]
\delta u(x,0)=\delta u_{t}(x,0)=0, \,x \in (0,\ell)\,; \\ [4pt]
\delta u(0,t)=\delta u_x(0,t)=0, ~\left(r(x)\delta u_{xx}\right)_{x=\ell}=\delta M(t),~\left (-(r(x)\delta u_{xx})_x\right)_{x=\ell}=\delta g(t),
\end{array} \right.
\end{eqnarray}
with the inputs $\delta M(t)=M_1(t)-M_2(t)$ and $\delta g(t)=g_1(t)-g_2(t)$.

We employ the trace inequalities (\ref{18}), applied to the weak solution of problem (\ref{25}), to estimate the norms:
\begin{eqnarray*}
\left. \begin{array}{ll}
\displaystyle \Vert \Phi(q_1)-\Phi(q_2) \Vert_{L^2(0,T)}=\Vert \delta u(\ell, \cdot) \Vert_{L^2(0,T)},  \\ [8pt]
\displaystyle \Vert \Psi(q_1)-\Psi(q_2) \Vert_{L^2(0,T)}=\Vert \delta u_x(\ell, \cdot) \Vert_{L^2(0,T)}
\end{array} \right.
\end{eqnarray*}
We have:
\begin{eqnarray*}
\left. \begin{array}{ll}
\displaystyle \Vert \Phi(q_1)-\Phi(q_2) \Vert_{L^2(0,T)} \le \frac{\ell \sqrt{\ell}}{\sqrt{2}}\,C_1\, \Vert \delta q \Vert_{\mathbb{H}^1(0,T)},  \\ [10pt]
\displaystyle \Vert \Psi(q_1)-\Psi(q_2) \Vert_{L^2(0,T)} \le \sqrt{\ell} \,C_1\, \Vert \delta q \Vert^2_{\mathbb{H}^1(0,T)}.
\end{array} \right.
\end{eqnarray*}
These estimates lead to the required inequalities (\ref{24}).\hfill$\Box$

%section 5
\section{The Tikhonov functional and existence of a quasi-solution}

In practice, exact equalities in the system of equations (\ref{23}) are not achievable due to random noise in the measured outputs $w_\ell(t)$ and $\theta_\ell(t)$. As a consequence, one needs to introduce the Tikhonov functional
\begin{eqnarray}\label{26}
\displaystyle \mathcal{J}(q):=\frac{1}{2} \int_0^T \left [\Phi(q)(t) -w_\ell(t) \right ]^2dt
+ \frac{1}{2} \int_0^T \left [\Psi(q)(t) -\theta_\ell(t) \right ]^2dt \qquad \qquad \qquad \nonumber \\ [6pt]
\qquad \equiv \frac{1}{2} \int_0^T \left [u(\ell,t;q) -w_\ell(t) \right ]^2dt
+ \frac{1}{2} \int_0^T \left [u_x(\ell,t;q) -\theta_\ell(t) \right ]^2dt\qquad \qquad \nonumber \\ [6pt]
=: \mathcal{J}_1(q)+ \mathcal{J}_2(q),~q \in \mathcal{Q}(0,T),\qquad  \quad
\end{eqnarray}
assuming that the measured outputs satisfy the following conditions:
\begin{eqnarray}\label{27}
w_\ell, \theta_\ell \in L^2(0,T).
\end{eqnarray}

Consider the following minimization problem for this functional:
 \begin{eqnarray}\label{28}
\mathcal{J}(q_*)=\inf_{q \in \mathcal{Q}(0,T)} \mathcal{J}(q).
\end{eqnarray}

A solution of the minimization problem (\ref{28}) is called a quasi-solution of the inverse source problem (\ref{1})-(\ref{2}), due to \cite{Ivanov:1962}.

The lemma below shows that the Lipschitz continuity property of the input-output operators yields the same property of the Tikhonov functional.

\begin{Lem}\label{Lemma-3}
Under the conditions of Theorem \ref{Theorem-1}, the Tikhonov functional (\ref{26}) is  Lipschitz continuous, that is,
\begin{eqnarray}\label{29}
\displaystyle \vert \mathcal{J}(q_1)-\mathcal{J}(q_2) \vert_{L^2(0,T)}
\leq  L_{\mathcal{J}}\, \Vert q_1-q_2 \Vert_{\mathbb{H}^1(0,T)}, ~\forall q_1,q_2 \in \mathcal{Q}(0,T),
\end{eqnarray}
where
\begin{eqnarray*}
\displaystyle L_{\mathcal{J}}=\sqrt{2}\,C_\mathcal{Q}\, \left (L_1^2+L_2^2 \right ) +L_1  \Vert w_\ell \Vert_{L^2(0,T)}+L_2 \Vert \theta_\ell \Vert_{L^2(0,T)},
\end{eqnarray*}
is the Lipschitz constant, $L_1$ and $L_2$ are the positive constants introduced in Lemma \ref{Lemma-2}.
\end{Lem}
{\bf Proof.} The following inequalities are obtained using the method employed in the proof of Lemma 11.2.3 in \cite{Hasanov:Romanov:2021}:
\begin{eqnarray*}
\displaystyle \vert \mathcal{J}_1(q_1)&-&\mathcal{J}_1(q_2) \vert \quad \qquad \qquad  \quad \qquad \qquad \qquad \qquad \qquad \qquad \qquad \quad \qquad \qquad \qquad \qquad \\ [2pt]
&\leq&  \frac{1}{2}\left\vert \int_0^T \left (u(\ell,t;q_1) -w_\ell(t) \right )^2dt-\int_0^T \left (u(\ell,t;q_2) -w_\ell(t) \right )^2dt  \right\vert   \\
&\leq& \frac{1}{2} \int_0^T \left\vert \Phi(q_1)-\Phi(q_2) \right\vert \cdot \left\vert \Phi(q_1)+\Phi(q_2)-2\,w_\ell \right\vert dt\\
&\leq& \frac{1}{2} \left(\int_0^T \left\vert \Phi(q_1)-\Phi(q_2) \right\vert^2 dt\right)^{1/2}\, \left(\int_0^T \left\vert \Phi(q_1)+\Phi(q_2) -2\, w_\ell\right\vert^2 dt\right)^{1/2}\\
&\leq&  \frac{1}{2}  \Vert \Phi(q_1)-\Phi(q_2) \Vert_{L^2(0,T)} \, \left(\Vert \Phi(q_1) \Vert_{L^2(0,T)}+\Vert \Phi(q_2) \Vert_{L^2(0,T)}+ 2\,\Vert w_\ell \Vert_{L^2(0,T)}  \right) \\
&\leq&  \frac{1}{2}\,\left( \Vert \Phi(q_1) \Vert_{L^2(0,T)} +\Vert \Phi(q_2) \Vert_{L^2(0,T)} +2\,\Vert w_\ell \Vert_{L^2(0,T)}\right)\, \Vert \Phi(q_1) -\Phi(q_1)  \Vert_{L^2(0,T)}.
\end{eqnarray*}
Due to (\ref{18}),
\begin{eqnarray*}
\displaystyle \vert \mathcal{J}_1(q_1)&-&\mathcal{J}_1(q_2) \vert \quad \qquad \qquad  \quad \qquad \qquad \qquad \qquad \qquad \qquad \qquad \quad \qquad \qquad \qquad \qquad \\ [2pt]
&\leq&  \frac{1}{2}\,\left[ \frac{\ell\sqrt{\ell}}{\sqrt{2}}C_1\left(\Vert q_1 \Vert_{\mathbb{H}^1(0,T)} +\Vert q_2 \Vert_{\mathbb{H}^1(0,T)}\right) +2\,\Vert w_\ell \Vert_{L^2(0,T)}\right]\, \Vert \Phi(q_1) -\Phi(q_1)  \Vert_{L^2(0,T)}. \end{eqnarray*}

Since  $\Vert q_i \Vert_{\mathbb{H}^1(0,T)}=\left(\Vert M \Vert^2_{H^1(0,T)}+\Vert g \Vert^2_{H^1(0,T)}\right)^{1/2}\leq \sqrt{2}\, C_\mathcal{Q}$ for $i=1,2$ and   (\ref{24}),
\begin{eqnarray*}
\displaystyle \vert \mathcal{J}_1(q_1)-\mathcal{J}_1(q_2) \vert \leq  \left[L_1^2\,\sqrt{2}\,C_\mathcal{Q} +L_1\,\Vert w_\ell \Vert_{L^2(0,T)}\right]\, \Vert q_1 -q_2  \Vert_{\mathbb{H}^1(0,T)}. 
\end{eqnarray*}
Similarly,
\begin{eqnarray*}
\displaystyle \vert \mathcal{J}_2(q_1)&-&\mathcal{J}_2(q_2) \vert \quad \qquad \qquad  \quad \qquad \qquad \qquad \qquad \qquad \qquad \qquad \quad \qquad \qquad \qquad \qquad \\ [2pt]
&\leq&  \frac{1}{2}\,\left[ \Vert \Psi(q_1) \Vert_{L^2(0,T)} +\Vert \Psi(q_2) \Vert_{L^2(0,T)} +2\Vert \theta_\ell \Vert_{L^2(0,T)}\right] \Vert \Psi(q_1) -\Psi(q_1)  \Vert_{L^2(0,T)}\\
&\leq& \left[L_2^2\,\sqrt{2}\,C_\mathcal{Q} +L_2\,\Vert \theta_\ell \Vert_{L^2(0,T)}\right]\, \Vert q_1 -q_2  \Vert_{\mathbb{H}^1(0,T)}. 
\end{eqnarray*}
Finally, using (\ref{26}), proof is completed.   \hfill$\Box$\\

This property of the Tikhonov functional provides the fundamental basis for the existence of a quasi-solution to the inverse problem (\ref{1})-(\ref{2}).

\begin{thm}\label{Theorem-3}
Assume that the conditions of Theorem \ref{Theorem-1} ensuring the existence of a weak solution of the forward problem hold. Suppose, in addition, that the measured outputs satisfy the conditions (\ref{27}). Then the minimization problem (\ref{28}) for  the Tikhonov functional (\ref{26}) has at least one solution in the set of admissible inputs $\mathcal{Q}(0,T)$. Hence, the inverse source problem (\ref{1})-(\ref{2}) has a quasi-solution.
\end{thm}

The proof of this theorem follows the same procedure as that of Theorem 10.1.11 in \cite{Hasanov:Romanov:2021}.

%section 6
\section{Fr\'echet differentiability of the Tikhonov functional}

In this section, we introduce first the formal Fr\'{e}chet gradient formula for the Tikhonov
functional through a weak solution of the related (unique) adjoint problem corresponding to the inverse problem (\ref{1})-(\ref{2}). Then, we will establish the sufficient conditions, including the conditions imposed on the measured outputs $w_\ell(t)$ and $\theta_\ell(t)$, for the existence of this gradient. Thus, the existence of the Fr\'{e}chet gradient of the Tikhonov functional will be justified.

First, we derive an integral identity establishing a relationship between the inputs $M(t)$ and $g(t)$ and outputs $u(\ell,t;q)$ and $u_x(\ell,t;q)$, including the measured outputs $w_\ell(t)$ and $\theta_\ell(t)$, through the weak solution of the related adjoint problem.

\begin{Lem}\label{Lemma-4}
Assume that the conditions of Theorem \ref{Theorem-1} are satisfied. Then between the inputs and outputs the following integral relationship holds:
\begin{eqnarray}\label{30}
\displaystyle \int_0^T \left \{\widehat {g}(t)\delta u(\ell,t) + \widehat {M}(t)\delta u_x(\ell,t) \right \}dt= \int_0^T \left \{\delta M(t)\phi_x(\ell,t)+\delta g(t)\phi(\ell,t) \right \} dt,
\end{eqnarray}
for all $\widehat {g},\widehat {M} \in H^1(0,T)$, where $\delta M(t)$ and $\delta g(t)$ are the Neumann inputs introduced in (\ref{25}), and the function $\phi (x,t)$ is the weak solution of the following backward problem:
\begin{eqnarray}\label{31}
\left\{ \begin{array}{ll}
\rho(x) \phi_{tt}-\mu(x)\phi_{t}+ (r(x)\phi_{xx})_{xx} =0,~(x,t)\in \Omega_{T};\\ [5pt]
\phi(x,T)=\phi_{t}(x,T)=0, ~x \in (0,\ell); \\ [5pt]
\phi(0,t)=\phi_x(0,t)=0,~\left(r(x)\phi_{xx}\right)_{x=\ell}=\widehat {M}(t),\, (-\left (r(x)\phi_{xx}\right)_{x})_{x=l}=\widehat {g}(t),\,t \in [0,T],
\end{array} \right.
\end{eqnarray}
with the inputs $\widehat {g},\widehat {M} \in H^1(0,T)$, and $\delta u(x,t)$  is the weak solution of the forward problem (\ref{24}).
\end{Lem}
{\bf Proof.} Multiply both sides of equation (\ref{25}) for $ \delta u (x,t)$ by  an arbitrary function $\phi(x,t)$, integrate over $\Omega_T$ and apply the integration by parts formula several times. Next, we get:
\begin{eqnarray}\label{32}
\displaystyle \int_0^T \int_0^{\ell} \left [\rho(x)\phi_{tt}-\mu(x)\phi_{t}+\left (r(x)\phi_{xx}\right)_{xx} \right] \delta u \,dx dt \qquad~\qquad \qquad~\qquad  \qquad~\qquad  \nonumber \\ [2pt]
\displaystyle \quad + \int_0^\ell \left [\rho(x) \delta u_t \phi-\rho(x) \delta u \phi_t +\mu(x) \delta u \phi \right ]_{t=0}^{t=T} \,dx \qquad~\qquad \qquad~\qquad  \qquad~\qquad \nonumber \\ [2pt]
\displaystyle + \int_0^T \left [\left(r(x) \delta u_{xx}\right)_{x}\phi-r(x)\delta u_{xx} \phi_{x}+
r(x)\phi_{xx}\delta u_{x}-\left (r(x)\phi_{xx}\right)_{x} \delta u  \right ]_{x=0}^{x=\ell} \,dt =0.~
\end{eqnarray}

The expressions in parentheses under the integrals in (\ref{32}) suggest to us how to choose an arbitrary function $\phi (x,t)$. Namely, the first of these expressions shows that this function must satisfy the adjoint equation (\ref{31}), and the other two expressions show that this function must satisfy the limit and final conditions in (\ref{31}). Thus, we assume that the function $\phi(x,t)$ solves the backward problem (\ref{31}). Then (\ref{32}) yields:
\begin{eqnarray*}
\displaystyle \int_0^T \left [-\delta g(t)\phi(\ell,t) -\delta M(t)\phi_x(\ell,t)
+\widehat {M}(t)\delta u_x(\ell,t)+\widehat {g}(t)\delta u(\ell,t) \right ]\,dt =0,~
\end{eqnarray*}
the required integral relationship (\ref{30}) is obtained. \hfill$\Box$

\begin{Cor}\label{Corollary-2}
Assume that the inputs $\widehat {M}(t)$ and $\widehat {g}(t)$ in the backward problem (\ref{31}) are  defined as follows
\begin{eqnarray}\label{33}
\left. \begin{array}{ll}
\displaystyle \widehat {M}(t)=u_x(\ell,t;q) -\theta_\ell(t), \\ [8pt]
\displaystyle \widehat {g}(t)=u(\ell,t;q) -w_\ell(t),\, t\in (0,T).
\end{array} \right.
\end{eqnarray}
The backward problem (\ref{31}) with the inputs given by (\ref{33}), i.e. the problem
\begin{eqnarray}\label{34}
\left\{ \begin{array}{ll}
\rho(x) \varphi_{tt}-\mu(x)\varphi_{t}+ (r(x)\varphi_{xx})_{xx} =0,~(x,t)\in \Omega_{T};\\ [5pt]
\varphi(x,T)=\varphi_{t}(x,T)=0, ~x \in (0,\ell); \\ [5pt]
\varphi(0,t)=\varphi_x(0,t)=0,~\left(r(x)\varphi_{xx}\right)_{x=\ell}=u_x(\ell,t;q) -\theta_\ell(t),\\ [3pt]
\qquad \qquad \qquad \qquad \qquad \left (-(r(x)\varphi_{xx}\right)_x)_{x=\ell}=u(\ell,t;q) -w_\ell(t),\,t \in [0,T],
\end{array} \right.
\end{eqnarray}
is defined as the adjoint problem corresponding to the inverse problem (\ref{1})-(\ref{2}). Furthermore, the inputs given by (\ref{33}) transform the integral identity (\ref{30}) into the following relationship
\begin{eqnarray}\label{35}
\displaystyle \int_0^T \left \{\left [u(\ell,t;q) -w_\ell(t) \right ]\delta u(\ell,t) + \left [u_x(\ell,t;q) -\theta_\ell(t) \right ]\delta u_x(\ell,t) \right \}dt  \qquad \quad \nonumber \\ [2pt]
= \int_0^T \left [\delta M(t)\varphi_x(\ell,t;q)dt+\delta g(t)\varphi(\ell,t;q) \right ] dt,~q \in \mathcal{Q}(0,T) \qquad
\end{eqnarray}
 which is defined as the input-output  relationship corresponding to the inverse problem (\ref{1})-(\ref{2}).
\end{Cor}

Let now $q, q+\delta q \in \mathcal{Q}(0,T)$. Denote  by $\delta J(q):= J(q+\delta q )-J(q)$ the increment of the Tikhonov functional introduced in (\ref{26}). Then,
\begin{eqnarray}\label{36}
\displaystyle \delta \mathcal{J}(q)= \int_0^T \left \{ \left [u(\ell,t;q) -w_\ell(t) \right ]\delta u(\ell,t;q)+ \left [u_x(\ell,t;q) -\theta_\ell(t) \right ]\delta u_x(\ell,t;q)\right \}dt \qquad \nonumber \\ [6pt]
+ \frac{1}{2} \int_0^T \left (\delta u(\ell,t;q) \right )^2dt
+ \frac{1}{2} \int_0^T \left (\delta u_x(\ell,t;q) \right )^2dt\qquad
\end{eqnarray}
Comparing the first right-hand-side integral in (\ref{36}) with the left-hand-side integral in the input-output  relationship (\ref{35}) we deduce that
\begin{eqnarray*}
\displaystyle \delta \mathcal{J}(q)=\int_0^T \left [\varphi_x(\ell,t;q)\delta M(t)+\varphi(\ell,t;q)\delta g(t) \right ]dt\qquad \qquad \quad \nonumber \\ [6pt]
+ \frac{1}{2} \int_0^T \left (\delta u(\ell,t;q) \right )^2dt
+ \frac{1}{2} \int_0^T \left (\delta u_x(\ell,t;q) \right )^2dt,
\end{eqnarray*}
or
\begin{eqnarray}\label{37}
\displaystyle \delta \mathcal{J}(q)=\int_0^T \left (\varphi_x(\ell,t;q),\varphi(\ell,t;q) \right )^T  \left (\delta M(t), \delta g(t)\right ) dt\qquad \qquad \qquad \qquad  \nonumber \\ [6pt]
+ \frac{1}{2} \int_0^T \left (\delta u(\ell,t;q) \right )^2dt
+ \frac{1}{2} \int_0^T \left (\delta u_x(\ell,t;q) \right )^2dt, ~q\in \mathcal{Q}(0,T).
\end{eqnarray}
This suggests that for the Fr\'echet gradient of the Tikhonov functional $\mathcal{J}(q)$ the following \emph{formal}  gradient formula holds:
\begin{eqnarray}\label{38}
\displaystyle \nabla \mathcal{J}(q)(t)=\left (\varphi_x(\ell,t;q),\varphi(\ell,t;q) \right )^T ,~q\in \mathcal{Q}(0,T)(0,T),\, t\in (0,T).
\end{eqnarray}

Establishing the formal gradient (\ref{38}) as the vector-form gradient of the Tikhonov functional requires more than proving that the last two right-hand-side integrals in (\ref{38}) are of the order  $\mathcal{O}\left (\Vert \delta q\Vert^2_{\mathbb{H}^2(0,T)}\right)$. Crucially, the substitution (\ref{33}) must be justified. Furthermore, Theorem 1 does not guarantee the existence of a weak solution to the adjoint problem defined by (\ref{33}) since the inputs $\displaystyle \widehat {M}(t)$ and $\displaystyle \widehat {g}(t)$ does not satisfy the regularity conditions in (\ref{3}).

\begin{thm}\label{Theorem-4}
Assume that conditions of Theorem \ref{Theorem-2} ensuring the existence of a regular weak solution of the forward problem (\ref{1}) hold. Suppose that the measured outputs $w_\ell(t)$ and $\theta_\ell (t)$ satisfy the following regularity conditions
\begin{eqnarray}\label{39}
w_\ell, \theta_\ell \in H^1(0,T).
\end{eqnarray}
Then the Tikhonov functional $\mathcal{J}(q)$, defined in (\ref{26}), is Fr\'echet differentiable. Moreover, for the Fr\'echet gradient $\nabla \mathcal{J}(q)$ of this functional  gradient formula (\ref{38}) holds.
\end{thm}
{\bf Proof.} Evidently, the final time adjoint problem (\ref{34}) is well-posed, as the transformation $\tau=T-t$ of the time variable shows. Further, under the conditions of Theorem \ref{Theorem-2} and the regularity conditions (\ref{39}), the Neumann inputs in the adjoint problem (\ref{34}) belong to $H^1(0,T)$, which means that conditions of Theorem \ref{Theorem-1} for this problem are satisfied. Hence by the theory developed in \cite{Baysal:2019}, there exists a unique weak solution to the adjoint problem (\ref{34}). Moreover, the a priori estimates derived in Theorem \ref{Theorem-1} hold for this solution also. In particular,
\begin{eqnarray}\label{40}
\displaystyle \Vert \varphi_{xx} \Vert^2_{L^2(0,T;L^2(0,\ell))}
\leq  C_1^2 \, \left [\Vert u_x(\ell,\,\cdot;q) -\theta_\ell \Vert^2_{H^1(0,T)}+
\Vert u(\ell,\,\cdot;q) -w_\ell \Vert^2_{H^1(0,T)}\right ].
\end{eqnarray}
In addition, the first right-hand-side integral in the first variation formula (\ref{37}) is well-defined.

Next, we consider the second and third right-hand-side integrals in formula (\ref{37}). It follows from the trace estimates (\ref{19}) applied to the regular weak solution $\delta u(x,t)$ of problem (\ref{25}) that
\begin{eqnarray}\label{41}
\left. \begin{array}{ll}
\displaystyle \Vert \delta u(\ell,\cdot) \Vert^2_{L^2(0,T)}
\leq   \frac{\ell^3 }{2}\, C_1^2 \, \Vert \delta q \Vert^2_{\mathbb{H}^2(0,T)}\,, \\ [10pt]
\displaystyle \Vert \delta u_{x} (\ell,\cdot) \Vert^2_{L^{\infty}(0,T)}
\leq \ell \,C_1^2 \, \Vert \delta q \Vert^2_{\mathbb{H}^2(0,T)}.
\end{array} \right.
\end{eqnarray}
In view of estimates (\ref{41}), formula (\ref{37}) can be rewritten as follows:
\begin{eqnarray*}
\displaystyle \delta \mathcal{J}(q)=\int_0^T \left (\varphi_x(\ell,t),\varphi(\ell,t) \right )  \left (\delta M(t), \delta g(t)\right )^T dt+\mathcal{O}\left (\Vert \delta q\Vert^2_{\mathbb{H}^2(0,T)}\right), \, q\in \mathcal{Q}(0,T).
\end{eqnarray*}
This justifies the Fr\'echet differentiability of the Tikhonov functional $\mathcal{J}(q)$, hence the gradient formula (\ref{38}).  \hfill$\Box$

%section 7
\section{The Lipschitz continuity of the Fr\'echet gradient}

It is well known that functionals with a Lipschitz continuous Fr\'echet gradient exhibit monotonicity when using gradient-based iteration algorithms. Therefore, the fact that the Fr\'echet gradient of the Tikhonov functional corresponding to any inverse problem is Lipschitz continuous is, in a sense, a natural regularizer.

\begin{thm}\label{Theorem-5}
If conditions of Theorem \ref{Theorem-2} are satisfied, then the Fr\'echet gradient of the Tikhonov functional $\mathcal{J}(q)$ corresponding to the inverse problem (\ref{1})-(\ref{2}) is Lipschitz continuous, that is
\begin{eqnarray}\label{42}
\displaystyle \Vert \nabla \mathcal{J}(q_1)-\nabla \mathcal{J}(q_2) \Vert_{\mathbb{L}^2(0,T)}
\leq  L_{\nabla}\, \Vert q_1-q_2 \Vert_{\mathbb{H}^1(0,T)}, ~\forall q_1,q_2 \in \mathcal{Q}(0,T),
\end{eqnarray}
where $\displaystyle L_{\nabla}=\left (\ell+ \frac {\ell^3}{2}\right ) C_1^2$  and
\begin{eqnarray*}
 \Vert \nabla \mathcal{J}(q_1)-\nabla \mathcal{J}(q_2) \Vert_{\mathbb{L}^2(0,T)}^2:= \Vert \delta \varphi_x(\ell,\cdot) \Vert^2_{L^2(0,T)}+\Vert \delta \varphi(\ell,\cdot) \Vert^2_{L^2(0,T)}.\\
\end{eqnarray*}
is the Lipschitz constant, $C_1>0$ is the constant introduced in Theorem \ref{Theorem-1}.
\end{thm}
{\bf Proof.} Denote by $\varphi(x,t;q_k)$ the weak solution of the adjoint problem (\ref{34}), corresponding to the admissible input $q_k \in \mathcal{Q}(0,T)$, $k=1,2$. Then the function $\delta \varphi(x,t):=\varphi(x,t;q_1)-\varphi(x,t;q_2)$ is the weak solution of the following problem:
\begin{eqnarray}\label{43}
\left\{ \begin{array}{ll}
\rho(x) \delta \varphi_{tt}-\mu(x)\delta \varphi_{t}+ (r(x)\delta \varphi_{xx})_{xx} =0,~(x,t)\in \Omega_{T};\\ [5pt]
\delta \varphi(x,T)=\delta \varphi_{t}(x,T)=0, ~x \in (0,\ell); \\ [5pt]
\delta \varphi(0,t)=\delta \varphi_x(0,t)=0,~\left(r(x)\delta \varphi_{xx}\right)_{x=\ell}=\delta u_x(\ell,t),\\ [3pt]
\qquad \qquad \qquad \qquad \qquad ~(-\left (r(x) \delta \varphi_{xx}\right)_x)_{x=\ell}=\delta u(\ell,t),\,t \in [0,T],
\end{array} \right.
\end{eqnarray}
where $\delta u(\ell,t)=u(x,t;q_1)-u(x,t;q_2)$ and $\delta u(x,t)$ is the regular weak solution of problem (\ref{25}). Furthermore,
\begin{eqnarray*}
\displaystyle \Vert \nabla \mathcal{J}(q_1)-\nabla \mathcal{J}(q_2) \Vert^2_{\mathbb{L}^2(0,T)}
=  \Vert \delta \varphi_x(\ell,\cdot) \Vert^2_{L^2(0,T)}+\Vert \delta \varphi(\ell,\cdot) \Vert^2_{L^2(0,T)}, ~q_1,q_2 \in \mathcal{Q}(0,T).
\end{eqnarray*}
To estimate the right-hand-side norms, we employ the trace estimates (\ref{18}) to get
\begin{eqnarray*}
\displaystyle \Vert \nabla \mathcal{J}(q_1)-\nabla \mathcal{J}(q_2) \Vert^2_{\mathbb{L}^2(0,T)}
\le \left (\ell+ \frac {\ell^3}{2}\right ) \,\Vert \delta \varphi_{xx} \Vert^2_{L^2(0,T;L^2(0,\ell))},
\end{eqnarray*}
and then use the first estimate in (\ref{4}) applied to the weak solution of problem (\ref{43}), to obtain the following inequality:
\begin{eqnarray*}
\displaystyle \Vert \nabla \mathcal{J}(q_1)-\nabla \mathcal{J}(q_2) \Vert^2_{\mathbb{L}^2(0,T)}
\le \left (\ell+ \frac {\ell^3}{2}\right ) C_1^2 \left [\Vert \delta u_x(\ell,\cdot) \Vert^2_{L^2(0,T)}+\Vert \delta u(\ell,\cdot) \Vert^2_{L^2(0,T)} \right ],
\end{eqnarray*}
In view of the trace inequalities (\ref{18}) applied to the function $\delta u(x,t)$ of problem (\ref{25}) this yields:
\begin{eqnarray*}
\displaystyle \Vert \nabla \mathcal{J}(q_1)-\nabla \mathcal{J}(q_2) \Vert^2_{\mathbb{L}^2(0,T)}
\le \left (\ell+ \frac {\ell^3}{2}\right )^2 C_1^2 \, \Vert \delta u_{xx} \Vert^2_{L^2(0,T;L^2(0,\ell))}.
\end{eqnarray*}
With the first estimate in (\ref{4}) applied to the solution $\delta u(x,t)$ of problem (\ref{25}) we arrive at the required result (\ref{42}) with $\delta q(t)= q_1(t)-q_2(t)$.
\hfill$\Box$

%%section 9
\section{Conclusion}
Based on previous studies, a new mathematical model has been developed that considers the influence of not only the shear force at the tip of the micro-cantilever but also the moment in a dynamic vibration in a micro-cantilever-sample system. Apart from Atomic Force Microscopy, these types of oscillating systems are indispensable components in the operation of various devices. Within this model, a new inverse problem has been formulated for determining the unknown shear force and the moment. This inverse problem assumes use of only the feasible and measurable outputs: deflection and slope at the tip of a micro-cantilever. This model allows for a more realistic simulation of vibration in micro-cantilever-sample contact than was previously possible. Within the weak solution theory for the forward problem, and the quasi-solution approach for the inverse problem combined with the Tikhonov functional and adjoint method, the theoretical basis for the proposed mathematical model has been developed. In this context, the vector form gradient formula has been derived which verifies the gradient formula given for the  simplified model \cite*{CHH:CCS:07}, and it has been proven that this gradient is Lipschitz continuous. The applications of these results, which form the basis for the numerical solution of the problem, will be discussed in the next study.

\section*{Acknowledgment}
The research of the first author has been supported by the Scientific and Technological Research Council of Turkey (TUBITAK) through the Incentive Program for International Scientific Publications (UBYT).\\

\end{document}